\newif\ifarxiv
\newif\ifwiley
\newif\ifshowchanges
\arxivtrue

\newcommand{\mytitle}{Robust estimation of a Markov chain transition matrix from multiple sample paths}

\newcommand{\WileyCopyright}{This is an open access article under the terms of the Creative Commons Attribution License, which permits use, distribution and reproduction in any medium, provided the original work is properly cited.
© 2026 The Author(s). Statistica Neerlandica published by John Wiley \& Sons Ltd on behalf of Netherlands Society for Statistics and Operations Research.}

\ifarxiv
\documentclass[11pt]{article}
\usepackage[authoryear]{natbib}
\usepackage[colorlinks=true, allcolors=blue]{hyperref}
\usepackage[a4paper,
  left=1in,
  right=1in,
  top=1.3in,
  bottom=1.3in
]{geometry}
\usepackage{amsmath,amsthm,amssymb}
\fi

\ifwiley
\documentclass[num-refs]{wiley-article}
\usepackage{color}
\usepackage{ulem}

\usepackage{siunitx}
\usepackage{comment}

\renewcommand{\margnote}[1]{\mbox{}\marginpar{\raggedright\hspace{0pt}{\color{purple} \tiny #1}}}

\papertype{Original Article}
\paperfield{Journal Section}

\title{\mytitle}

\author[1\authfn{1}]{Lasse Leskel\"a}
\author[2\authfn{1}]{Maximilien Dreveton}

\affil[1]{Department of Mathematics and Systems Analysis, Aalto University, Espoo, 02015, Finland}
\affil[2]{Information and Network Dynamics Laboratory, EPFL, Lausanne, Switzerland}

\corraddress{Lasse Leskel\"a, Department of Mathematics and Systems Analysis, Aalto University,
Espoo, 02015, Finland}
\corremail{lasse.leskela@aalto.fi}

\runningauthor{Leskel\"a and Dreveton}

\fi

\ifwiley
\usepackage[utf8]{inputenc}
\usepackage{lslbasicmath}
\usepackage{lslnames}
\usepackage{lslabbrv}
\usepackage{lslhats}
\fi

\ifarxiv
\usepackage{enumitem}
\usepackage{datetime}
\newdateformat{mydate}{\THEDAY \ \monthname[\THEMONTH] \THEYEAR}
\mydate
\usepackage[utf8]{inputenc}
\usepackage{lslbasicmath}
\usepackage{lslnames}
\usepackage{lslabbrv}
\usepackage{lslhats}
\usepackage{lslcomm}
\fi

\numberwithin{equation}{section}

\ifarxiv
\theoremstyle{plain}
\newtheorem{theorem}{Theorem}[section]
\newtheorem{proposition}[theorem]{Proposition}
\newtheorem{lemma}[theorem]{Lemma}
\newtheorem{corollary}[theorem]{Corollary}
\newtheorem*{theorem*}{Theorem}
\theoremstyle{definition}

\newtheorem{remark}[theorem]{Remark}
\newtheorem{example}[theorem]{Example}
\newtheorem*{remark*}{Remark}
\fi

\renewcommand{\vec}{\operatorname{vec}}

\renewcommand{\spenorm}[1]{\lVert#1\rVert_2}

\newcommand{\Wcol}{W^{\rm col}}
\newcommand{\Wrow}{W^{\rm row}}

\newcommand{\gamin}{\ga_{\rm min}}
\newcommand{\garev}{\ga_{\rm rev}}

\newcommand{\barpi}{\bar\pi}
\newcommand{\barpimin}{\barpi_{\rm min}}

\newcommand{\pimin}{\pi_{\rm min}}

\newcommand{\pimink}{\pi^*_k}
\newcommand{\piminzero}{\pi^*_0}

\newcommand{\tC}{\tilde C}

\ifshowchanges
\newcommand{\newtext}[1]{{\color{purple}#1}}
\newcommand{\margnote}[1]{\mbox{}\marginpar{\raggedright\hspace{0pt}{\color{purple} \tiny #1}}}
\else
\newcommand{\newtext}[1]{#1}
\newcommand{\margnote}[1]{}
\fi

\usepackage{graphicx}
\usepackage{caption}
\usepackage{subcaption}

\begin{document}

\ifarxiv
\title{\mytitle\thanks{\WileyCopyright}}
\date{22 Jan 2026}
\author{
Lasse Leskelä\thanks{Corresponding author}\\[.5ex]
\small Aalto University\\
\small \texttt{lasse.leskela@aalto.fi}
\and
Maximilien Dreveton\\[.5ex]
\small Universit{\'e} Gustave-Eiffel \\
\small \texttt{maximilien.dreveton@univ-eiffel.fr}
}
\maketitle
\fi

\begin{abstract}
Markov chains are fundamental models for stochastic dynamics, with applications in a wide range of areas such as population dynamics, queueing systems, reinforcement learning, and Monte Carlo methods. Estimating the transition matrix and stationary distribution from observed sample paths is a core statistical challenge, particularly when multiple independent trajectories are available. While classical theory typically assumes identical chains with known stationary distributions, real-world data often arise from heterogeneous chains whose transition kernels and stationary measures might differ from a common target. We analyse empirical estimators for such parallel Markov processes and establish sharp concentration inequalities that generalise Bernstein-type bounds from standard time averages to ensemble-time averages. Our results provide nonasymptotic error bounds and consistency guarantees in high-dimen\-sion\-al regimes, accommodating sparse or weakly mixing chains, model mismatch, nonstationary initialisations, and partially corrupted data. These findings offer rigorous foundations for statistical inference in heterogeneous Markov chain settings common in modern computational applications.
\end{abstract}

\ifwiley
\fi

\ifarxiv
\bigskip
\fi

\ifarxiv
\small
\noindent{Published as:}
\emph{Statistica Neerlandica} 80(1), e70023, 2026. \url{https://doi.org/10.1111/stan.70023}
\fi

\section{Introduction}

Markov chains are fundamental models for stochastic dynamics, with applications ranging from population and queueing models \citep{Diekmann_Heesterbeek_Britton_2013,Meyn_Tweedie_2009} to 
Monte Carlo methods \citep{Andrieu_Lee_Vihola_2018,Brown_2025,Margossian_etal_2024} and
reinforcement learning \citep{Bardenet_Doucet_Holmes_2017,Chandak_Borkar_Dodhia_2022,Qu_Wierman_2020}.
A key task in such settings is the estimation of the transition probability matrix from observed data \citep{Craig_Sendi_2002,Mostel_Pfeuffer_Fischer_2020,TrendelkampSchroer_etal_2015}.
This problem has a long history: classical works such as \citep{Anderson_Goodman_1957,Billingsley_1961}
studied the limiting properties of estimators under long time horizon. More recent research has focused
on nonasymptotic bounds that provide finite-sample guarantees,
particularly for sparse models and high-dimensional settings \citep{Chatterjee_2025,Hao_Orlitsky_Pichapati_2018,Lezaud_1998,Wolfer_2024,Wolfer_Kontorovich_2019,Wolfer_Kontorovich_2021}.
However, this literature almost exclusively considers data from a single chain, typically assumed to be stationary and well-specified.

In contrast, many practical settings involve heterogeneous sources with multiple chains, where the transition dynamics and stationary behaviour may vary across trajectories and differ from a shared target model with transition matrix $P$.
This arises in applications like cohort studies, epidemiology, and distributed systems. We study the problem of estimating $P$ and its stationary distribution $\pi$ from $M$ such paths of length $T$.
A natural approach aggregates observed transitions and state visits into empirical estimators $\hat{P}$ and $\hat{\pi}$. While well understood for homogeneous, stationary chains, their performance under heterogeneity remains largely unexplored.

We develop a general theoretical framework for this setting, quantifying how estimation errors depend on heterogeneity in transition matrices, initial distributions, and mixing properties. We derive nonasymptotic error bounds and consistency guarantees that hold even in high-dimensional regimes, where the number of states, trajectories, and trajectory lengths may all grow. Our results are particularly suited to sparse or slowly mixing chains and allow for partial data corruption.
This work extends several foundational results from the single-trajectory literature---such as concentration inequalities for empirical averages
\citep{Huang_Li_2024+,Miasojedow_2014,Paulin_2015}
and transition matrix estimation
\citep{Huang_Li_2025,Wolfer_Kontorovich_2019,Wolfer_Kontorovich_2021}---to the setting of \emph{ensemble-time averages} over multiple heterogeneous trajectories. In particular, we generalise a Paulin--Bernstein concentration inequality \citep{Paulin_2015}, and the analysis of the empirical transition matrix in \citep{Wolfer_Kontorovich_2019}, from the single-trajectory setting to the multi-trajectory regime.

The work presented in this article was originally motivated by the study of community detection in temporal networks,
in which interactions between intra- and inter-community node pairs evolve over time according to binary Markov chains with transition matrices $P$ and $Q$, respectively~\citep{Avrachenkov_Dreveton_Leskela_2024, Avrachenkov_Dreveton_Leskela_2025+}.
\newtext{
A widely used clustering procedure follows a three-stage approach:\margnote{Numbering changed to inline.}
(i) an initial (usually noisy) assignment of community labels,
(ii) estimation of interaction parameters via empirical averages conditioned on the initial partition,
(iii) refinement of the partition using likelihood-based updates \citep{Abbe_Bandeira_Hall_2016,Gao_Ma_Zhang_Zhou_2017}.
}
In this setting, errors from stage (i) propagate into stage (ii), contaminating the empirical estimation process. As a result, the empirical estimator for~$P$ incorporates interaction sequences not only from correctly identified intra-community node pairs but also from misclassified inter-community pairs, which in fact follow the dynamics of the matrix~$Q$. This highlights the need for a robust understanding of empirical estimators under model mismatch and heterogeneous trajectory sources—a regime that is the focus of this work.

The rest of the paper is organised as follows.\margnote{Updated structure}
\newtext{
Section~\ref{sec:Model} describes the model, and Section~\ref{sec:MainResults} summarises the main results. Section~\ref{sec:NumericalExperiments} presents the results of numerical experiments, and Section~\ref{sec:Conclusions} concludes. All proofs are contained in Appendix~\ref{sec:ProofAppendix}, with the core part in Appendix~\ref{sec:MarkovConcentration}, which contains a concentration analysis of ensemble–time averages that may be of independent interest.
}

\section{Model description}
\label{sec:Model}

Fix integers $M, T \ge 1$, and let $\Ome$ be a finite set. The objective is to estimate the transition matrix $P$ of a Markov chain on $\Ome$ from $M$ independent sample paths, that are arranged in a data matrix $(X_{m,t})$
with rows indexed by $m = 1, \dots, M$ and columns by $t = 0, \dots, T$.

A natural estimator for $P$ is the \emph{empirical transition matrix} $\hat{P} \in \R^{\abs{\Ome} \times \abs{\Ome}}$, defined by
\begin{equation}
\label{eq:EmpiricalTransitionMatrix}
\hat{P}_{i,j} \weq
\begin{cases}
\displaystyle\frac{N_{i,j}}{N_i} & \text{if } N_i > 0, \\
\displaystyle\frac{1}{\abs{\Ome}} & \text{otherwise},
\end{cases}
\end{equation}
where the state and transition counts are given by
\begin{equation}
\label{eq:MarkovFrequencies}
N_i = \sum_{m=1}^M \sum_{t=1}^{T} X_{m,t-1}^{(i)},
\qquad
N_{i,j} = \sum_{m=1}^M \sum_{t=1}^{T} X_{m,t-1}^{(i)} X_{m,t}^{(j)},
\end{equation}
with indicator variables $X_{m,t}^{(i)} = 1(X_{m,t} = i)$.
Furthermore, if $P$ is irreducible and aperiodic, its stationary distribution $\pi$ may be estimated by the \emph{empirical distribution}
\begin{equation}
\label{eq:EmpiricalDistribution}
\hat{\pi}_i \weq \frac{N_i}{MT}.
\end{equation}
In the setting where all sample paths are realisations of a Markov chain with transition matrix $P$,
it is well known that $\hat{P}$ coincides with the maximum likelihood estimate \citep{Anderson_Goodman_1957}.

To investigate the robustness of these estimators under model misspecification, we consider a more general setting in which the observed sample paths are independent but not identically distributed. Specifically, each row of the data matrix $(X_{m,t})$ is a realisation of an irreducible and aperiodic Markov chain with
transition matrix $P_m \in \R^{\abs{\Ome} \times \abs{\Ome}}$,
initial distribution $\mu_m$,
stationary distribution $\pi_m$, and
pseudo-spectral gap $\ga_m$.
We expect the estimators \eqref{eq:EmpiricalTransitionMatrix} and \eqref{eq:EmpiricalDistribution}
to perform well when $P_m \approx P$ and $\pi_m \approx \pi$ for most $m$,
and when the total sample size $MT$ is large.\margnote{Old Sec 2.3 merged to updated Sec 3.3}

\section{Main results}
\label{sec:MainResults}
This section contains the main results.
Section~\ref{sec:NonasymptoticErrorBoundsNew} presents general nonasymptotic error 
bounds.
Section~\ref{sec:FullyCorrupted} extends the analysis to a setting where
some rows of the data matrix are fully corrupted.
Section~\ref{sec:Consistency} presents large-scale consistency results 
obtained as corollaries.

\subsection{Nonasymptotic error bounds}
\label{sec:NonasymptoticErrorBoundsNew}

\newtext{Recall that each row of the data matrix $(X_{m,t})$} \margnote{Shortened}
is a Markov chain on $\Omega$ with
transition matrix $P_m$, initial distribution $\mu_m$,
and 
stationary distribution $\pi_m$.
To quantify how close
\newtext{the transition matrices governing the rows are from the target matrix $P$,
we denote}\margnote{Removed unused symbols $\de_1,\de_\infty, \bar\pi_{\rm max}$}
\begin{equation}
 \label{eq:Deltas}
 \begin{aligned}
 \Delta_1       = \frac{1}{M} \sum_m \| P_m - P \|_\infty, 
 \qquad \Delta_\infty  = \max_m \| P_m - P \|_\infty.
 \end{aligned}
\end{equation}
We also denote $\barpimin = \min_i \newtext{\barpi(i)}$\margnote{Updated $\barpi_i \mapsto \barpi(i)$}
where 
$
\newtext{\barpi(i)} = \frac{1}{M} \sum_m \pi_m(i)
$
is the long-term average frequency of state $i$,
and the minimum pseudo-spectral gap by $\gamin = \min_m \ga_m$.
The deviation from stationarity is quantified by
\begin{equation}
 \label{eq:RenyiTwo}
 \eta
 \weq \frac{1}{M} \sum_m D_2( \mu_m \| \pi_m ),
\end{equation}
where $D_2$ refers to the \Renyi divergence of order two (see Appendix~\ref{sec:InformationTheory} for details).

\newtext{
The mixing properties\margnote{Definition moved here, to be visible earlier.}
of the Markov chains are conveniently described in terms of a normalised time parameter 
\begin{equation}
 \label{eq:TPrime}
 T'
 \weq \frac{\gamin T}{1 + 1/(\gamin T)},
\end{equation}
where $\gamin = \min_m \ga_m$ denotes the minimum pseudo-spectral gap of the transition matrices $P_m$.
The quantity $T' \le T$ corresponds to an effective time that is discounted
by the worst-case mixing behaviour of the model, and is closely
related to the mixing times of Markov chains \cite[Proposition 3.4]{Paulin_2015}.
The quantity $M T'$ then corresponds to the number of well-mixed samples.
We also note that
$T' \asymp T$ in well-mixing regimes with $\gamin \asymp 1$.
}

The following theorem presents a generic error bound for estimating the transition matrix $P$
using the empirical transition matrix $\hat P$ defined in \eqref{eq:EmpiricalTransitionMatrix}.

\begin{theorem}\margnote{Statement streamlined}
\label{the:TransitionMatrixError}
\newtext{
There exist universal constants \(C_1, C_2, C_3 > 0\) such that for any \(0 < \eps \le 1\),
the transition matrix estimator satisfies
\begin{equation}
 \label{eq:TransitionMatrixError}
 \supnorm{\hhP - P}
 \wle C_1 \sqrt{ \frac{\abs{\Omega} \log(4\abs{\Omega}/\eps)}{\barpimin MT} }
 + C_2 \min \left\{ \frac{\De_1}{\barpimin}, \, \De_\infty \right\}
\end{equation}
with probability at least $1-\eps$,
provided the effective sample size satisfies
\begin{equation}
 \label{eq:TransitionMatrixErrorLargeSample}
 MT' \wge C_3 \frac{\log(4\abs{\Omega}/\eps) + M \eta}{\barpimin}.
\end{equation}
}
\end{theorem}

Furthermore, the following theorem presents an error bound for the empirical distribution $\hat \pi$ defined in \eqref{eq:EmpiricalDistribution}.

\begin{theorem}\margnote{Statement streamlined}
\label{the:StationaryDistributionError}
\newtext{
There exists a universal constant $C_1 > 0$
such that for any $0 < \eps \le 1$,
the stationary distribution estimator satisfies
\begin{equation}
 \label{eq:StationaryDistributionError}
 \supnorm{\hat\pi - \bar\pi}
 \wle C_1 \sqrt{ \frac{\log (2\abs{\Omega}/\eps) + M\eta}{MT'} }
\end{equation}
with probability at least $1-\eps$.
}
\end{theorem}

\begin{example}[Clean data]
\newtext{
Consider a data matrix $(X_{m,t})$ in which the rows are
mutually independent, irreducible and aperiodic, stationary Markov chains with
transition matrix $P$,
stationary distribution $\pi$,
and pseudo-spectral gap $\ga$.
Then $\barpi = \pi$, and the deviation terms 
$\De_1, \De_\infty, \eta$ vanish.
Then the bounds \eqref{eq:TransitionMatrixError} and \eqref{eq:StationaryDistributionError}
become\margnote{Updated to match new theorem formulation}
\[
 \supnorm{\hhP - P}
 \wle C_1 \sqrt{ \frac{\abs{\Omega} \log(4\abs{\Omega}/\eps)}{\pimin MT} }
 \qquad \text{and} \qquad
 \supnorm{\hat \pi - \pi}
 \wle C_1 \sqrt{ \frac{\log (2\abs{\Omega}/\eps)}{MT'} },
\]
and they are valid, 
provided the effective sample size satisfies
$
 MT' \ge C_3 \frac{\log(4\abs{\Omega}/\eps)}{\pimin}.
$
}
\end{example}

\newtext{
By combining~\eqref{eq:StationaryDistributionError} with the triangle inequality,
we see that the estimation error for the distribution $\pi$ is bounded by
\begin{equation}
 \label{eq:StationaryDistributionErrorTriangle}
 \supnorm{\hat\pi - \pi}
 \wle C_1 \sqrt{ \frac{\log (2\abs{\Omega}/\eps) + M\eta}{MT'} } + \supnorm{\barpi - \pi}.
\end{equation}
}

\begin{example}[\newtext{Mixture of correct and misspecified chains}]\margnote{Descriptive name added}
Assume that each row of the data matrix corresponds to a stationary Markov chain with transition matrix
and stationary distribution
\[
 P_m \weq
 \begin{cases}
 P, & m \in \mathcal{M}_0, \\
 Q, & m \in \mathcal{M}_1,
 \end{cases}
 \qquad\text{and}\qquad
 \pi_m \weq
 \begin{cases}
 \pi, & m \in \mathcal{M}_0, \\
 \rho, & m \in \mathcal{M}_1,
 \end{cases}
\]
where \(P\) denotes the correct transition matrix,
and \(Q\) is a misspecified (perturbed) transition matrix.
\newtext{
In this case
$\De_1 = \frac{M_1}{M} \supnorm{P-Q}$, 
$\De_\infty = \supnorm{P-Q}$,
and
$\supnorm{\bar\pi - \pi} = \frac{M_1}{M} \supnorm{\pi-\rho}$,
where $M_1 = |\mathcal{M}_1|$ is the number of perturbed chains.
Then the bounds \eqref{eq:TransitionMatrixError} and \eqref{eq:StationaryDistributionErrorTriangle}
become\margnote{Updated to match new theorem formulation}
\[
 \supnorm{\hhP - P}
 \wle C_1 \sqrt{ \frac{\abs{\Omega} \log(4\abs{\Omega}/\eps)}{\barpimin MT} }
 + C_2 \min \left\{ \frac{M_1/M}{\barpimin}, \, 1 \right\} \supnorm{P-Q}
\]
and
\[
 \supnorm{\hat \pi - \pi}
 \wle C_1 \sqrt{ \frac{\log(2\abs{\Omega}/\eps)}{MT'} }
 + \frac{M_1}{M} \supnorm{\pi-\rho}.
\]
}
\end{example}

\subsection{Robust estimation under fully corrupted rows}
\label{sec:FullyCorrupted}

Let us consider a setting in which $M_0$ rows of the data matrix $(X_{m,t})$ are Markov chains
as before, and the remaining $M_1 = M-M_0$ rows are arbitrary, considered as fully corrupted data.
When we do not know which of the rows are corrupted, a natural estimate of $P$ is still
the empirical transition matrix $\hat P$ of the full data matrix given by \eqref{eq:EmpiricalTransitionMatrix}.
When $M_1/M$ is small, we expect that $\hat P$ still provides a reasonably accurate estimate of $P$.
The following result quantifies when this is the case.

As before, we assume that among the $M_0$ rows,
each row $m$ is an irreducible and aperiodic Markov chain with initial distribution $\mu_m$,
transition matrix $P_m$,
stationary distribution $\pi_m$, and pseudo-spectral gap $\ga_m$.
We define \newtext{
$\De_1^{(0)}, \De_\infty^{(0)}, \barpimin^{(0)}$
}
as in \eqref{eq:Deltas},\margnote{superscript $( \cdot )^{(0)}$ added}
and \newtext{$\eta^{(0)}$} as in \eqref{eq:RenyiTwo}, 
with the averages and maxima taken over the $M_0$ uncorrupted rows.

\newtext{
\begin{theorem}\margnote{Statement streamlined}
\label{the:TransitionMatrixErrorCorrupted}
There exist universal constants $C_1,C_2,C_3,C_4 > 0$ such that
for any $0 < \eps \le 1$,
the transition matrix estimator satisfies
\begin{equation}
 \label{eq:TransitionMatrixErrorFullyCorruptedNew}
 \supnorm{\hat P - P}
 \wle C_1 \sqrt{ \frac{\abs{\Omega} \log(8\abs{\Omega}/\eps)}{\barpimin^{(0)} M_0 T} }
 + C_2 \min \left\{ \frac{\De_1^{(0)}}{\barpimin^{(0)}}, \, \De_\infty^{(0)} \right\}
 + C_3  \frac{M_1/M}{\barpimin^{(0)}}
\end{equation}
with probability at least $1-\eps$, whenever
\begin{equation}
 \label{eq:CondTransitionMatrixErrorFullyCorruptedNew}
 M_0 T' \wge C_4 \frac{\log(8\abs{\Omega}/\eps) + M_0 \eta^{(0)}}{\big(\barpimin^{(0)}\big)^2}.
\end{equation}
\end{theorem}
}

\begin{remark}
\newtext{
The proofs of the main results yield explicit admissible constants.
One may take\margnote{New remark clarifying the role of explicit constants.}
$(C_1,C_2,C_3) = (7,2,144)$ in Theorem~\ref{the:TransitionMatrixError},
$C_1= \sqrt{56}$ in Theorem~\ref{the:StationaryDistributionError},
and $(C_1,C_2,C_3,C_4) = (7,2,4,144)$ in Theorem~\ref{the:TransitionMatrixErrorCorrupted}.
We do not expect the explicit admissible constants to be sharp,
although we do expect the asymptotic rates with respect to $M,T'$ to be optimal.
}
\end{remark}

\subsection{Consistency in large-scale regimes}
\label{sec:Consistency}

\newtext{\margnote{Old Section 2.3 merged here}
To study asymptotic consistency, we introduce a scale parameter $\nu \to \infty$
and allow all model parameters and dimensions to depend on $\nu$,
including the size and shape of the state space $\Omega$,
the number and length of chains $M,T$,
the transition matrices $P, P_m$,
and the distributions $\pi, \pi_m, \mu_m$.
For positive sequences $a = a_\nu$ and $b=b_\nu$
indexed by $\nu$, we write
$a \ll b$ to mean $\lim_{\nu \to \infty} a_\nu / b_\nu = 0$, and
$a \lesim b$ to mean $\limsup_{\nu \to \infty} a_\nu / b_\nu < \infty$.
We also write $a \asymp b$ to indicate that $a \lesim b$ and $b \lesim a$.
For notational simplicity, we omit explicit dependence on $\nu$.

This scaling framework is significant because it encompasses a broad spectrum of asymptotic regimes, including large ensembles ($M \to \infty$) of long chains ($T \to \infty$), growing state spaces ($|\Omega| \to \infty$), and sparse models where $\pi(i) \to 0$ or $P(i,j) \to 0$ for most $i,j$. These regimes are relevant for understanding the behaviour of the estimators in both high-dimensional and heterogeneous settings.
}

As immediate corollaries of Theorems~\ref{the:TransitionMatrixError} and \ref{the:StationaryDistributionError}, we obtain sufficient conditions for $\hat P$ and $\hat \pi$ to be consistent estimators of $P$ and $\pi$,
recalling the definition of the effective time $T'$ in \eqref{eq:TPrime}.

\begin{corollary}[Transition matrix consistency]
\label{the:TransitionMatrixConsistency}
If 
\[
MT' \gg \frac{|\Omega| \log |\Omega|}{\barpimin}, \qquad 
T' \gg \frac{\eta}{\barpimin}, \qquad 
\min\Big\{ \frac{\Delta_1}{\barpimin}, \Delta_\infty \Big\} \ll 1,
\]
\newtext{
then $\supnorm{\hat P - P} \to 0$ in probability as $\nu \to \infty$.\margnote{Updated statement, less technical}
}
\end{corollary}

\begin{corollary}[Stationary distribution consistency]
\label{the:StationaryDistributionConsistency}
If
\[
MT' \gg \log |\Omega|, \qquad
T' \gg \eta, \qquad
\supnorm{\bar\pi - \pi} \ll 1,
\]
\newtext{
then $\supnorm{\hat \pi - \pi} \to 0$ in probability as $\nu \to \infty$.\margnote{Updated statement, less technical}
}
\end{corollary}

In a stationary setting ($\eta = 0$) with bounded state space ($|\Omega| \asymp 1$), Corollary~\ref{the:TransitionMatrixConsistency} requires
$\barpimin MT' \gg 1$ and $\min\{\Delta_1/\barpimin, \Delta_\infty\} \ll 1$.
A lower bound on $\barpimin$ is needed because estimating $P_{i,:}$ becomes difficult if some states are rarely visited. In contrast, Corollary~\ref{the:StationaryDistributionConsistency} only requires $MT' \gg 1$ and $\supnorm{\bar\pi - \pi} \ll 1$, so no lower bound on $\barpimin$ is needed: rare states are naturally estimated close to zero.

In balanced regimes with $\barpimin \asymp 1$, $\min\{\Delta_1/\barpimin, \Delta_\infty\} \asymp \Delta_1$. In contrast, in sparse regimes with $\barpimin \ll 1$, it may happen that $\Delta_\infty \ll 1 \lesssim \Delta_1 / \barpimin$,
as shown in the following example.

\begin{example}[\newtext{Random walk on the complete graph}]\margnote{Descriptive name added}
\label{exa:CompleteGraph}
Let $\Omega = \{1,\dots,n\}$ for $n \ge 3$, with 
\[
P(i,j) = \frac{1}{n}, \qquad P_m(i,j) = \frac{1(j \ne i)}{n-1}.
\]
Then $P$ and $P_m$ are the transition matrices of symmetric random walks on the complete graph
where $P$ allows self-loops but $P_m$ forbids them.
Both $P$ and $P_m$ are irreducible and aperiodic,
with uniform stationary distribution $\pi_m = \pi$.
A simple computation gives $\supnorm{P_m - P} = 2/n$, $\barpimin = 1/n$,
and $\Delta_1 = \Delta_\infty = 2/n$. Hence, $\Delta_\infty \ll 1 \asymp \Delta_1 / \barpimin$ as $n \to \infty$.
\end{example}

\section{\newtext{Numerical experiments}}
\label{sec:NumericalExperiments}

We evaluate the\margnote{New section}
empirical transition matrix estimator through controlled simulations, validating theoretical error bounds and illustrating the trade-offs in ensemble-based estimation.

\paragraph{Experimental setup.}
Consider a lazy random walk on a cycle of $\abs{\Omega}$ vertices. From each vertex $i$, the chain moves to either neighbour with probability $\ga/2$ and stays at $i$ with probability $1-\ga$.
The transition matrix of this reversible Markov chain is
\[
 P(i,i\pm 1 \bmod \abs{\Omega})
 \weq \frac{\ga}{2}, \qquad P(i,i) = 1-\ga.
\]
Standard computations \cite[Section 12.3]{Levin_Peres_2017_Markov}
show that for any $\ga \le \frac12$,
the pseudo-spectral gap of $P$ equals
$\ga_{\rm ps} = 1 - (1 - \ga_{\rm abs})^2$,
where $\ga_{\rm abs} = \ga(1-\cos(2\pi /\abs{\Omega}))$ is the absolute spectral gap.
Then $\ga_{\rm ps} = \frac{4\pi^2}{\abs{\Omega}^2} \ga + O\big(\frac{1}{\abs{\Omega}^4} \big)$ for $\abs{\Omega} \gg 1$.

To assess robustness under model heterogeneity, we generate $M$ perturbed transition matrices, each obtained by adding i.i.d.\ uniform noise in $(-\eps, \eps)$ to the entries of $P$, truncating negative values to zero, and renormalising rows to ensure stochasticity.
Alternative perturbation schemes (e.g., noise applied only to diagonal entries or to a subset of elements per row) yielded qualitatively similar results and are omitted for brevity.

\paragraph{Trade-off between chain count and trajectory length.}

In the first experiment (Figure~\ref{fig:VaryingChainCountStateCount}, left)
we fix the total number of observations at $MT=10^4$ and vary how they are split between
the number of chains $M$ and their length $T$.
In the absence of noise, the estimation error remains roughly constant.
With noise, the error decreases as the number of chains increases up to about 100, then levels off.
When $M$ is small, the mean deviation $\Delta_1$ fluctuates considerably, as a few perturbed chains dominate the average, preventing concentration around the target $P$.
Averaging over more trajectories reduces this variability, improving estimator accuracy, in line with our theoretical predictions.

\begin{figure}[!h]
\centering
\includegraphics[height=55mm, trim=0 4mm 0 0, clip]
{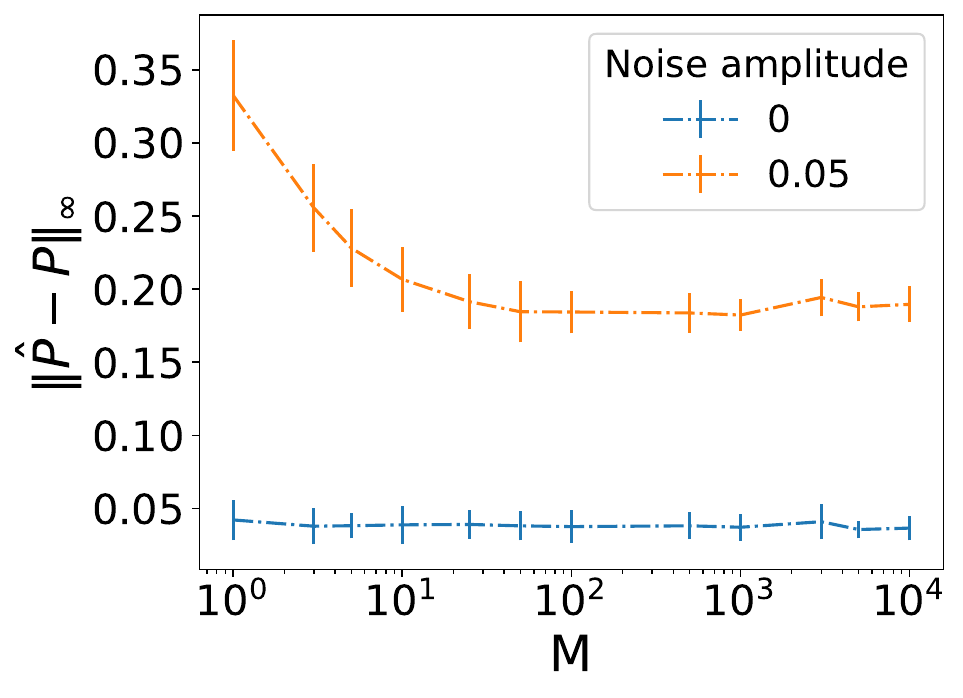}
\hspace{2mm}
\includegraphics[height=55mm, trim=0 4mm 0mm 0, clip]
{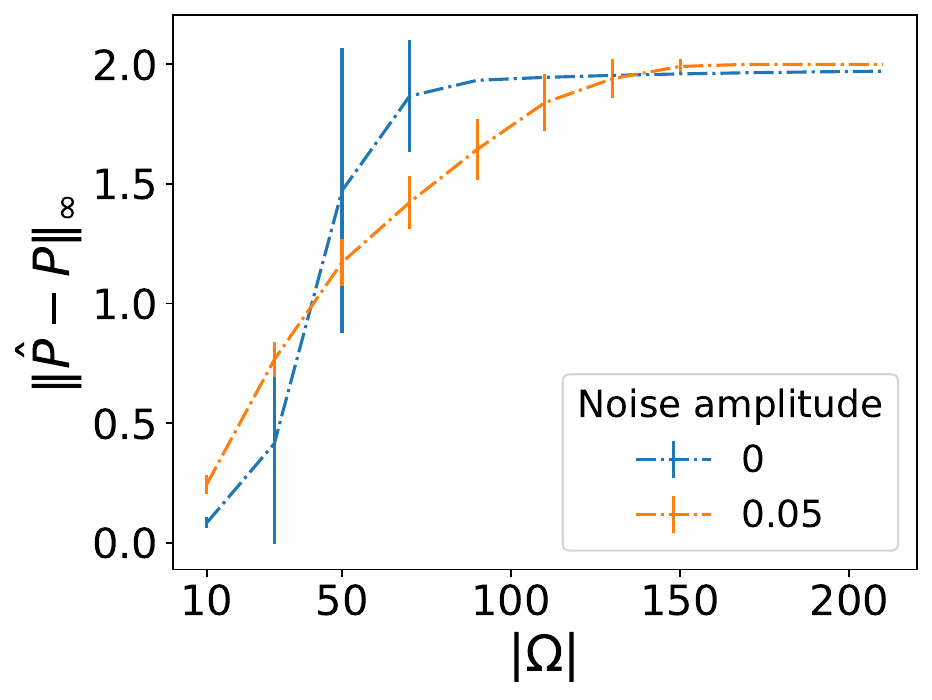}
\caption{
Transition matrix estimation error for
varying chain count $M$ (left) with respect to fixed total sample count $MT=10^4$ 
and state count $\abs{\Omega}=10$; and
varying state count $\abs{\Omega}$ (right) with $M=50$ and $T=50$.
Results are averaged over 50 simulations, and
error bars display standard deviations for
two noise levels ($\eps=0$ and $\eps=0.05$) with $\ga=0.1$.
\label{fig:VaryingChainCountStateCount}
}
\end{figure}

\paragraph{Effect of state count.}
Figure~\ref{fig:VaryingChainCountStateCount} (right) examines the influence of the state space size $\abs{\Omega}$ on the estimation error. With $M = 50$, $T = 50$, and $\gamma = 0.1$ fixed, the error initially increases with $\abs{\Omega}$,
then plateaus at its theoretical upper bound of 2.
This bound follows from the fact that both $P$ and $\hat{P}$ are stochastic matrices, so
$\supnorm{\hat{P} - P} \le 2$ by the triangle inequality.
This behaviour is consistent with Theorem~\ref{the:TransitionMatrixError}, which predicts that,
for fixed total observations $MT$,
the estimation error grows with the square root of the state count $\abs{\Omega}$.

\paragraph{Influence of spectral gap.}
Figure~\ref{fig:VaryingGap} shows how the jump rate $\gamma$, which controls the spectral gap, affects estimation error for $|\Omega| = 10$, $M = 200$, and $T = 200$. When $\gamma$ is small, the chain mixes slowly and successive samples are highly correlated, increasing the error of the stationary distribution in both noiseless and noisy regimes.
For the transition-matrix estimator, the behaviour is subtler. For $\gamma$ close to 0, the transition matrix is nearly the identity; a single long trajectory provides no additional information, so accurate estimation relies on multiple independent chains. Even small perturbations then produce large deviations in the noisy setting. As $\gamma$ increases, the dynamics become more complex, and for $\gamma$ close to 1 the chain approaches periodic behaviour, also leading to larger errors. In the noiseless regime, estimation errors grow moderately with $\gamma$, remaining small overall. In the noisy regime, the smallest errors occur when $\gamma$ and the perturbation amplitude $\eps$ are comparable, while errors increase when $\gamma$ is very small or close to 1.

\begin{figure}[!h]
\centering
\includegraphics[height=53mm, trim=0 3mm 0 0, clip]
{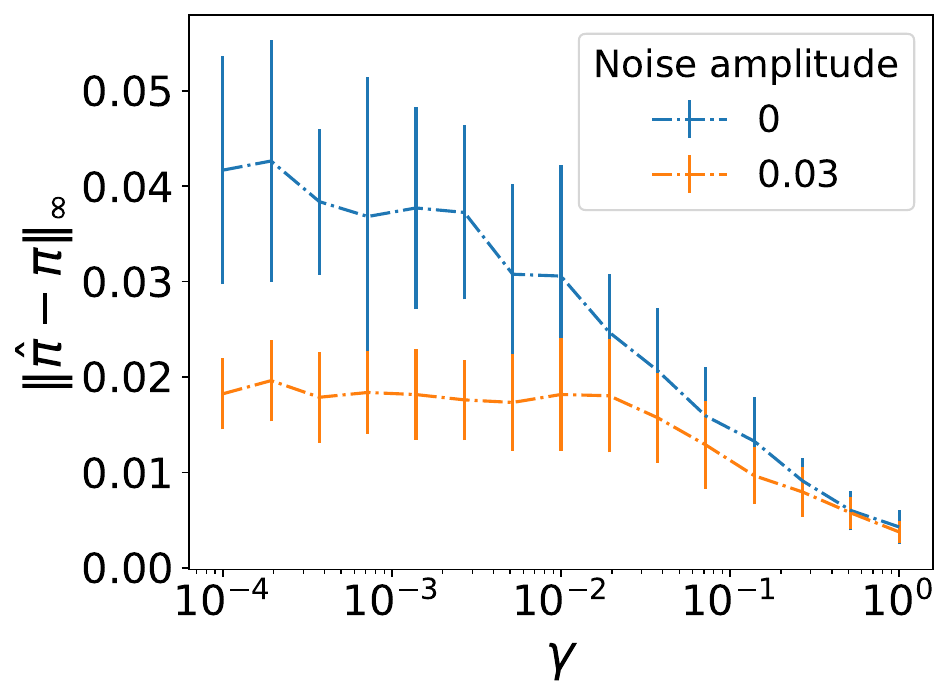}
\hspace{6mm}
\includegraphics[height=53mm, trim=0 3mm 0 0, clip]
{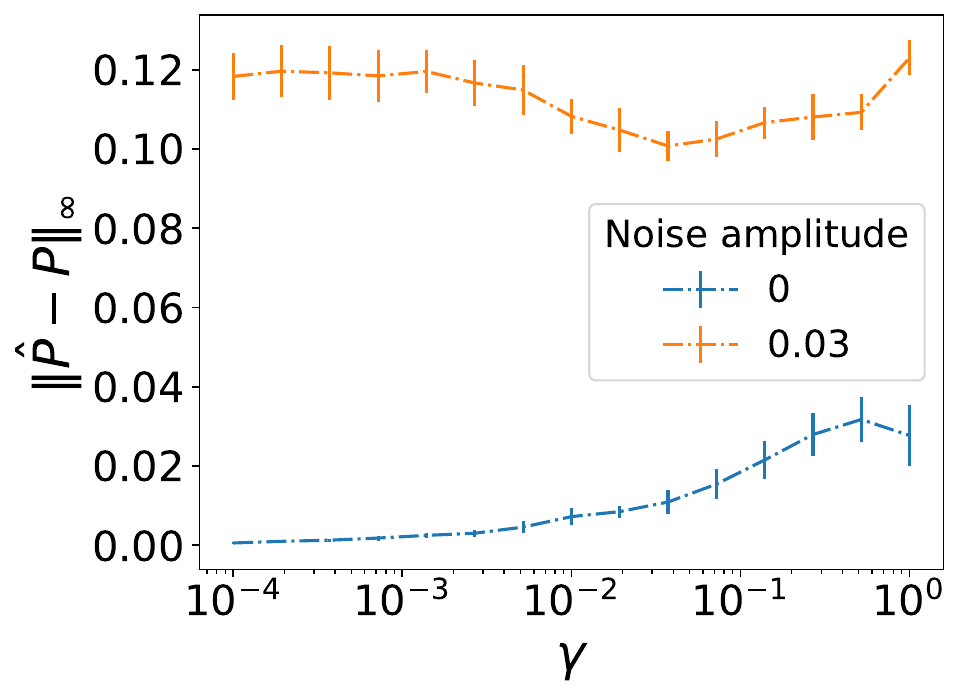}
\caption{
Stationary distribution (left) and transition matrix (right) estimation error for
varying transition rate parameter $\ga$,
with 
$|\Omega| = 10$, $M=200$, $T=200$.
Results are averaged over 50 simulations, and
error bars display standard deviations for
two noise levels ($\eps=0$ and $\eps=0.03$).
\label{fig:VaryingGap}
}
\end{figure}

\section{Conclusions}
\label{sec:Conclusions}

This article established consistency results and explicit nonasymptotic error bounds for estimating the transition matrix of a Markov chain from multiple parallel sample paths. We showed that a simple estimator—the empirical transition matrix—is consistent under mild assumptions on the chains’ mixing properties. This aligns with known results for estimation from a single sample path \citep{Huang_Li_2025,Wolfer_Kontorovich_2019,Wolfer_Kontorovich_2021}.

A key advance of this work is the ability to rigorously quantify the trade-off between the number of observed paths ($M$) and their length ($T$), thereby generalizing classical settings. In addition, we developed a robust estimation framework that accounts for perturbations in the dynamics and allows for the presence of fully corrupted sample paths.

Several important directions remain open for future research. These include extending the analysis to continuous-time Markov processes, e.g.\ by following the ideas in \citep{Huang_Li_2024+}, and to positive recurrent chains on countably infinite state spaces. Robust parameter estimation in such settings would be valuable for applications such as performance analysis of queueing systems and forecasting in epidemiological models. Furthermore, these developments could provide a foundation for tackling more complex scenarios involving partially observed Markov chains, commonly encountered for example in distributed queueing systems \citep{Leskela_2006,Leskela_Resing_2007} and stochastic epidemic models \citep{Britton_2020,Diekmann_Heesterbeek_Britton_2013}.

%
%

\section{Proofs}
\label{sec:ProofAppendix}

\newtext{This appendix is organised as follows.\margnote{New structure (with all proofs in appendix)}
Section \ref{sec:Notation} introduces technical notation.
Section \ref{sec:MarkovConcentration} provides
concentration inequalities for Markov chain ensembles.
Section \ref{sec:Proofs} contains the proofs of the main results.
Section \ref{sec:Technical} contains auxiliary technical results.
}

\subsection{Notation}
\label{sec:Notation}

Functions $\mu \colon \Ome \to \R$ are identified as vectors $\mu \in \R^{\abs{\Ome}}$
and functions $P \colon \Ome \times \Ome \to \R$ as matrices $P \in \R^{\abs{\Ome} \times \abs{\Ome}}$.
Notations $\mu_i = \mu(i)$ and $P_{i,j} = P(i,j)$ are used interchangeably as dictated by notational clarity and convenience.
Standard vectors norms are denoted by
$\onenorm{\mu} = \sum_i \abs{\mu_i}$,
$\twonorm{\mu} = (\sum_i \mu_i^2)^{1/2}$,
and
$\supnorm{\mu} = \max_i \abs{\mu_i}$.
For matrices, $\norm{P}_2$ denotes the spectral norm, and
$\norm{P}_\infty = \max_i \norm{P_{i,:}}_1$ the maximum row sum norm;
these are the operator norms induced by the $\ell_2$ and $\ell_\infty$ distances on $\R^{\abs{\Ome}}$, respectively \citep{Horn_Johnson_2013}.
For stochastic matrices, $\norm{P-Q}_\infty = 2 \max_i d_{\rm tv}(P_{i,:}, Q_{i,:})$
equals twice the largest rowwise total variation distance.

Standard notations concerning Markov chains are used \citep{Levin_Peres_2017_Markov}.
A vector $\mu \in \R^{\abs{\Ome}}$ is called \emph{stochastic}
if $\mu_i \ge 0$ for all $i$ and $\sum_i \mu_i = 1$.
A matrix $P \in \R^{\abs{\Ome} \times \abs{\Ome}}$ is called
\emph{stochastic} if its rows are stochastic.
A stochastic matrix is called \emph{irreducible} if for any $i,j$ there exists an integer
$k \ge 1$ such that $(P^k)_{i,j} > 0$.
A stochastic vector $\pi$ is \emph{stationary} for $P$ if $\pi P = \pi$.
Any irreducible stochastic matrix $P$ admits a unique stationary vector $\pi$.
The \emph{time reversal} of such matrix is defined by $P^*_{i,j} = \frac{\pi_j}{\pi_i} P_{j,i}$. 
An irreducible stochastic matrix is \emph{reversible} if $P^* = P$.
The \emph{spectral gap} of an irreducible reversible stochastic matrix $P$
is defined by $\garev(P) = 1-\la_2(P)$ where $\la_2(P)$ denotes the
second largest eigenvalue of $P$.
The \emph{pseudo-spectral gap} of an irreducible stochastic matrix $P$ is defined by
\[
 \ga(P)
 \weq \sup_{k \ge 1} \frac{\garev(P^{*k} P^k)}{k}.
\]
See \cite{Paulin_2015} for details, and \cite{Huang_Li_2025} for alternative notions of spectral gaps
for nonreversible stochastic matrices.


\subsection{Concentration of Markov chain ensembles}
\label{sec:MarkovConcentration}

In this section we derive concentration inequalities for empirical state frequencies
(Proposition~\ref{the:StateFrequenciesInd}) and empirical transition frequencies (Proposition~\ref{the:TransitionFrequencies})
of Markov chain ensembles that are key to proving the main results of Section~\ref{sec:MainResults}.
The starting point is a Paulin--Bernstein concentration inequality (Theorem~\ref{the:StateFrequencies})
for ensemble-time averages of Markov chains that is proved in Section~\ref{sec:PaulinParallelMarkov}.
The proof of Proposition~\ref{the:StateFrequenciesInd} is then given in Section~\ref{sec:ConcentrationStateFrequencies},
and the proof of Proposition~\ref{the:TransitionFrequencies} in 
Section~\ref{sec:ConcentrationTransitionFrequencies}.

\subsubsection{Concentration of ensemble-time averages}
\label{sec:PaulinParallelMarkov}

We will generalise a Paulin--Bernstein concentration inequality \cite[Theorem 3.4, Proposition 3.10]{Paulin_2015}
for the standard time average
$\frac{1}{T} \sum_{t=1}^T f(X_{t-1})$ of a single Markov chain
to the ensemble-time average 
$\frac{1}{MT} \sum_{m=1}^M \sum_{t=1}^T f(X_{m,t-1})$
of multiple independent Markov chains.

\begin{remark}
\label{rem:PaulinVector}
A simple approach for deriving a concentration bound is to represent the ensemble-time average
as an ordinary time average of the $\Ome^M$-valued Markov chain $(X_{:,t})$ 
according to
\[
 \frac{1}{MT} \sum_{m=1}^M \sum_{t=1}^T f(X_{m,t-1})
 \weq \frac{1}{T} \sum_{t=1}^T F(X_{:,t-1}),
\]
where $F(x_1,\dots, x_M) = \sum_{m=1}^M f_m(x_m)$, and then apply
\cite[Theorem 3.4]{Paulin_2015}.
Such a direct approach does not lead into a good concentration inequality
for large ensembles with $M \gg 1$.
The reason is that, even though the variance $\Var_\Pi(F) = \sum_{m=1}^M \Var_{\pi_m}(f_m)$
for $\Pi = \bigotimes_m \pi_m$ being the stationary distribution of $(X_{:,t})$
scales well with $M$, the same is not in general true for the term $\Delta_F = \max_{x \in \Ome^M} \abs{F(x) - \Pi(F)}$.
In particular, for indicator functions $f_m = 1_{\{i\}}$, we see that $\Delta_F$ is of order $M$
even though $\max_x \abs{f_m(x) - \pi_m(f_m)} \le 1$.
\end{remark}

Instead of the aforementioned simplistic approach, we will refine the analysis and
modify the proof of \cite[Theorem 3.4]{Paulin_2015}
to properly account for the smoothing effect due to averaging over multiple independent Markov chains.
As a result, we obtain the following sharp concentration inequality.

\begin{theorem}
\label{the:StateFrequencies}
For any functions $f_1,\dots, f_M \colon \Ome \to \R$
and for any $s > 0$,
the ensemble-time average
$S = \frac{1}{MT} \sum_{m=1}^M \sum_{t=1}^T (f_m(X_{m,t-1}) - \pi_m(f_m))$ satisfies
\begin{equation}
 \label{eq:StateFrequenciesF}
 \pr( S \ge s)
 \wle \exp \bigg( - \frac{\gamin MT s^2}{16 (1+1/(\gamin T)) V + 40 \De s} + \frac12 M \eta \bigg),
\end{equation}
where $\De = \max_m \max_x \abs{f_m(x) - \pi_m(f_m)}$ and
$V = \frac{1}{M} \sum_m \sum_x ( f_m(x) - \pi_m(f_m))^2 \pi_m(x)$.
\end{theorem}

\begin{proof}
Let us first prove \eqref{eq:StateFrequenciesF} in the stationary case in which
the rows $X_{m,:}$ are stationary Markov chains.  In this case we denote
the law of the random matrix $(X_{m,t})$ by $\pr_\pi$.
Denote
$\De_m = \max_x \abs{f_m(x)-\pi_m(f_m)}$
and
$\Var_{\pi_m}(f_m) = \sum_x (f_m(x)-\pi_m(f_m))^2 \pi_m(x)$,
and define
\[
 g_m(x)
 \weq
 \begin{cases}
 \frac{f_m(x) - \pi_m(f_m)}{\De_m}, &\quad \De_m \ne 0, \\
 0, &\quad \text{otherwise}.
 \end{cases}
\]
Then $S = \frac{1}{MT} \sum_m \De_m S_m$, where
$
 S_m
 = \sum_{t=1}^ T g_m(X_{m,t}).
$

Fix an integer $m$ such that $\De_m > 0$.
We note that $\pi_m(g_m) = 0$ and $\supnorm{g_m} = 1$.
The last inequality in the proof of \cite[Theorem 3.4]{Paulin_2015} then implies that
\[
 \E_\pi e^{\theta S_m} \le e^{\psi_m(\theta)}
 \qquad \text{for all $0 < \theta < \ga_m/10$},
\]
where
\[
 \psi_m(\theta)
 \weq 2 \Var_{\pi_m}(g_m) \frac{T+1/\ga_m}{\ga_m} \frac{\theta^2}{1 - \frac{10 \theta}{\ga_m}}
 \weq 2 \Var_{\pi_m}(f_m) \frac{T+1/\ga_m}{\De_m^2 \ga_m} \frac{\theta^2}{1 - \frac{10 \theta}{\ga_m}}.
\]
Then
\[
 \E_\pi e^{\theta \De_m S_m} \le e^{\psi_m(\De_m \theta)}
 \qquad \text{for all $0 < \theta < \frac{\ga_m}{10 \De_m}$}.
\]

Denote $\eps = \min_{m: \De_m > 0} \frac{\ga_m}{\De_m}$.
Then for any $0 < \theta < \eps/10$, we find that
\[
 \E_\pi e^{\theta MT S}
 \weq \prod_{m: \De_m > 0} \E_\pi e^{\theta \De_m S_m}
 \wle e^{\sum_{m: \De_m > 0} \psi_m(\De_m \theta)}.
\]
Here
\begin{align*}
 \sum_{m: \De_m > 0} \psi_m(\De_m \theta)
 &\weq 2 \sum_{m: \De_m > 0} \Var_{\pi_m}(f_m) \frac{T+1/\ga_m}{\De_m^2 \ga_m}
  \frac{\De_m^2 \theta^2}{1 - \frac{10 \De_m \theta}{\ga_m}} \\
 &\wle 2 \underbrace{\sum_m \Var_{\pi_m}(f_m) \frac{T+1/\ga_m}{\ga_m}}_{\al}
  \frac{\theta^2}{1 - \frac{10 \theta}{\eps}}.
\end{align*}
Hence by Markov's inequality, for any $0 < \theta < \eps/10$ and all $s \ge 0$,
\begin{equation}
 \label{eq:PaulinNormalisedTensor0Gen}
 \pr_\pi( S \ge s/T )
 \weq \pr_\pi( MTS \ge M s )
 \wle e^{-\theta Ms} \E_\pi e^{\theta MTS}
 \wle \exp \bigg( M (\psi(\theta) - \theta s) \bigg),
\end{equation}
where
$
 \psi(\theta)
 = 2 \frac{\al}{M} \frac{\theta^2}{1 - \frac{10 \theta}{\eps}}.
$
To minimise the right side in \eqref{eq:PaulinNormalisedTensor0Gen}, we note that
\[
 \psi(\theta) - \theta s
 \weq \frac{\al \eps^2}{50 M} \left( \frac{\eta^2}{1 - \eta} - s_1 \eta \right),
\]
where $\eta = \frac{10\theta}{\eps}$ and
$s_1 = 5 \frac{Ms}{\al \eps}$.
The function $\psi_1(\eta) = \frac{\eta^2}{1 - \eta} - s_1 \eta$
has derivatives
$\psi_1'(\eta) = \frac{2\eta - \eta^2}{(1-\eta)^2} - s_1$
and
$\psi_1''(\eta) = \frac{2}{(1-\eta)^3}$.
Hence $\psi_1$ is strictly convex on $(0,1)$
and attains its minimum at $\eta = 1 - \frac{1}{\sqrt{1+s_1}}$.
As a consequence, $\psi(\theta) - \theta s$ is minimised at
$\theta = \frac{\eps \eta}{10}$.
A simple computation shows that the minimum value equals
\begin{align*}
 \psi(\theta) - \theta s
 &\weq - \frac{\al \eps^2}{50 M} \left( \sqrt{1+s_1} -1 \right)^2.
\end{align*}

To simplify the above expression, we apply the inequality\footnote{This can be checked by squaring things so that we see a difference between 4th order polynomials.
This gives a shaper bound that the simple Taylor bound 
$(1+x)^{1/2} - 1 = \int_0^x \frac12 (1+s)^{-1/2} ds \ge \frac12 \frac{x}{(1+x)^{1/2}}$.}
$\sqrt{1+x} - 1 \ge \frac{x/2}{\sqrt{1+x/2}}$
which implies that $( (1+x)^{1/2} - 1)^2 \ge \frac{x^2/4}{{1+x/2}}$.
We conclude that
\begin{align*}
 \psi(\theta) - \theta s
 \wle - \frac{\al \eps^2}{50 M} \frac{s_1^2/4}{{1+s_1/2}}
 \weq - \frac{ \eps s^2}{8 \frac{\al \eps}{M} + 20 s}.
\end{align*}
By plugging this into \eqref{eq:PaulinNormalisedTensor0Gen}, 
we see that
\begin{equation}
 \label{eq:StateFrequenciesF1}
 \pr_\pi( S \ge s/T )
 \wle \exp \bigg( - \frac{M \eps s^2}{8 \frac{\al \eps}{M} + 20 s} \bigg).
\end{equation}
Then observe that $\eps \ge \frac{\gamin}{\De}$ and
\[
 \al
 \weq \sum_m \Var_{\pi_m}(f_m) \frac{T+1/\ga_m}{\ga_m}
 \wle \frac{T+1/\gamin}{\gamin} M V,
\]
with $V = \frac{1}{M} \sum_m \Var_{\pi_m}(f_m)$,
so that
\[
 \frac{M \eps s^2}{8 \frac{\al \eps}{M} + 20 s}
 \wge \frac{M s^2}{8 ( \frac{T+1/\gamin}{\gamin} V ) + 20 s \frac{\De}{\gamin}}
 \weq \frac{\gamin M s^2}{8 (T+1/\gamin) V + 20 s \De}.
\]
By plugging the above inequality into \eqref{eq:StateFrequenciesF1} and substituting $sT$ in place of $s$,
we find that
\begin{equation}
 \label{eq:StateFrequenciesFGStationary}
 \pr_\pi( S \ge s )
 \wle \exp \bigg( - \frac{\gamin M T s^2}{8 (1+1/(\gamin T)) V + 20 \De s} \bigg).
\end{equation}

Finally, let us prove the claim in the general nonstationary context.
By Lemma~\ref{the:NonstationaryMC}, we see that
\[
 \pr(S \ge s)
 \wle e^{\frac12 \sum_m D_2( \mu_m \| \pi_m )} \ \pr_\pi(S \ge s)^{1/2}
 \weq e^{\frac12 M \eta} \ \pr_\pi(S \ge s)^{1/2}
\]
By combining this with \eqref{eq:StateFrequenciesFGStationary}, the claim follows.
\end{proof}

\subsubsection{Concentration of empirical state frequencies}
\label{sec:ConcentrationStateFrequencies}

\begin{proposition}
\label{the:StateFrequenciesInd}
For any state $i$
and any $s > 0$,
the empirical frequency $N_i = \sum_{m=1}^M \sum_{t=1}^T X_{m,t-1}^{(i)}$
of state~$i$ satisfies
\[
 \pr\bigg( \Bigabs{\frac{N_i}{MT} - \barpi_i} \ge s \bigg)
 \wle 2 \exp \bigg( - \frac{\gamin MT s^2}{16 (1+1/(\gamin T)) \barpi_i + 40 s} + \frac12 M \eta \bigg).
\]
\end{proposition}
\begin{proof}
Fix $f_m(x) = 1_{\{i\}}(x)$ and define
$S = \frac{1}{MT} \sum_{m=1}^M \sum_{t=1}^T (f_m(X_{m,t-1}) - \pi_m(f_m))$.
Then 
$S = \frac{N_i}{MT} - \barpi_i$.
Note that
\[
 \De
 \weq \max_m \max_x \abs{f_m(x) - \pi_m(f_m)}
 \wle 1
\]
and
\[
 V
 \weq \frac{1}{M} \sum_m \sum_x ( f_m(x) - \pi_m(f_m))^2 \pi_m(x)
 \weq \frac{1}{M} \sum_m \pi_m(i)(1-\pi_m(i))
 \wle \barpi_i.
\]
By applying
Theorem~\ref{the:StateFrequencies}, we find that
\[
 \pr( S \ge s )
 \wle \exp \bigg( - \frac{\gamin MT s^2}{16 (1+1/(\gamin T)) \barpi_i + 40 s} + \frac12 M \eta \bigg).
\]
By repeating the same argument for $f_m(x) = - 1_{\{i\}}(x)$, we see that the above bound is valid 
also for $\pr( S \le - s )$.
Hence the claim follows.
\end{proof}

\subsubsection{Concentration of transition frequencies}
\label{sec:ConcentrationTransitionFrequencies}

In this section we will derive a concentration inequality for the empirical transition probabilities
out from a particular state $i$, given that the empirical frequency of visiting $i$ is concentrated
to an interval.  This analysis extends a martingale concentration argument in \cite{Wolfer_Kontorovich_2019} 
from a single Markov chain to a heterogeneous ensemble of Markov chains.

Denote the state and transition frequencies of the chain $m$ by
\begin{equation}
 \label{eq:MarkovFrequenciesIndiv}
 N_{m,i} = \sum_{t=1}^{T} X_{m,t-1}^{(i)}
 \qquad\text{and}\qquad
 N_{m,i,j} = \sum_{t=1}^T X_{m,t-1}^{(i)} X_{m,t}^{(j)}
\end{equation}
and let
\begin{equation}
 \label{eq:MeanTransitionMatrix}
 \tilde P_{i,j}
 \weq
 \begin{cases}
   \sum_m \frac{N_{m,i}}{N_i} P_{m,i,j} &\quad \text{if $N_i > 0$}, \\
   \frac{1}{\abs{\Ome}} &\quad \text{else},
 \end{cases}
\end{equation}
where $P_{m,i,j} = P_m(i,j)$ is the $(i,j)$-entry of the transition matrix $P_m$.

\begin{proposition}
\label{the:TransitionFrequencies}
Let $(X_{m, t})$ be a random array in $\Ome^{M \times (T+1)}$
in which the rows are independent, and
each row is a Markov chain with transition matrix $P_m \in \R^{\abs{\Ome} \times \abs{\Ome}}$.
Then the empirical transition matrix $\hat P$ computed by \eqref{eq:EmpiricalTransitionMatrix}
is concentrated around $\tilde P$ defined by \eqref{eq:MeanTransitionMatrix}
according to
\begin{align*}
 &\pr( \norm{\hat P_{i,:} - \tilde P_{i,:}}_1 \ge \eps,  \, s_1 \le N_i \le s_2) \\
 &\wle (1+\abs{\Ome}) \exp \bigg( - \frac{3 \eps^2 s_1}{6 \sqrt{2} \abs{\Ome} s_2/s_1 + 2 \sqrt{2} \abs{\Ome}^{1/2} \eps} \bigg)
\end{align*}
for all $i \in \Ome$, $\eps > 0$, and $0 < s_1 \le s_2$.
\end{proposition}

Again, a simple approach would be to apply a martingale concentration argument to a vector-valued Markov chain
$(X_{:,t})$, but this would lead to a suboptimal bound due to similar issues as in Remark~\ref{rem:PaulinVector}.
To properly incorporate the concentration properties due to the two dimensions of averaging, both
over the ensemble members $m$ and the time slots $t$.
Instead, we will here study an alternative martingale corresponding to the time series obtained by
vectorising the data matrix $(X_{m,t})$ by stacking its columns on top of each other.

The starting point for proving Proposition~\ref{the:TransitionFrequencies} is to note that
\[
 N_{i,j} = \sum_{m=1}^M \sum_{t=1}^T X_{m,t-1}^{(i)} X_{m,t}^{(j)}
 \qquad\text{and}\qquad
 N_i \tilde P_{i,j} = \sum_{m=1}^M \sum_{t=1}^TX_{m,t-1}^{(i)} P_{m,i,j}.
\]
Then we may write 
\begin{equation}
 \label{eq:CenteredTransitionCount}
 N_{i,j} - N_i \tilde P_{i,j}
 \weq \sum_{m=1}^M \sum_{t=0}^T U_{m,t,i,j}
\end{equation}
in terms of
\begin{equation}
 \label{eq:UDef}
 U_{m,t,i,j}
 \weq
 \begin{cases}
 0, &\qquad t=0, \\
 X_{m,t-1}^{(i)} (X_{m,t}^{(j)} - P_{m,i,j}), &\qquad t \ge 1.
 \end{cases}
\end{equation}
The zeros are added to obtain a more convenient indexing, with 
$U_{:,:,i,j} \in \R^{M \times (T+1)}$ having the dimensions
as the data array $X$.
The challenge is to represent these indicators as martingale differences.
To resolve this, we change the indexing of time as follows.
Let
\[
 \barU_{:,i,j} = \vec(U_{:,:,i,j})
\]
be a vector
obtained by stacking the columns of the matrix $U_{:,:,i,j}$ on top of each other.
Then $\barU_{:,i,j}$ is a random vector in $\R^K$ with $K = M(T+1)$
which shall be viewed as a random process $k \mapsto \barU_{k,i,j}$
with time parameter $k$.
We define an associated cumulative process by
\begin{equation}
 \label{eq:CumulativeProcess}
 \barV_{k,i,j}
 \weq \sum_{k'=1}^k \barU_{k',i,j},
 \qquad k=0,\dots,K,
\end{equation}
with the empty sum considered as zero. Equality \eqref{eq:CenteredTransitionCount} can
then be written as
\begin{equation}
 \label{eq:CenteredTransitionCountTranslated}
 \barV_{K,i,:}
 \weq N_{i,:} - N_i \tilde P_{i,:}.
\end{equation}

The random process $k \mapsto \barV_{k,i,:}$ is adapted to the filtration
$(\cF_0, \cF_1,\dots)$
of sigma-algebras
$\cF_k = \sig(\barX_1, \dots \barX_k)$ in which $\barX_k$ indicates
the $k$-th element of the vectorised data matrix $\barX = \vec(X)$.
The next result confirms that the process is a martingale with bounded increments.

\begin{lemma}
\label{the:Martingale}
The random process $k \mapsto \barV_{k,i,:}$ defined by \eqref{eq:CumulativeProcess} is a martingale with respect to
$(\cF_0, \cF_1, \dots)$,
with increments bounded by
$\norm{ \barU_{k,i,:} }_2 \le \sqrt{2}$
almost surely.
\end{lemma}

\begin{proof}
Define
\begin{equation}
 \label{eq:YDef}
 Y_{m,t,i,j}
 \weq
 \begin{cases}
 0, &\qquad t=0, \\
 X_{m,t-1}^{(i)} X_{m,t}^{(j)}, &\qquad t \ge 1.
 \end{cases}
\end{equation}
Define $\barY_{:,i,j} = \vec(Y_{:,:,i,j})$.
The entries of these matrices and their vectorisations are related by
\[
 Y_{m,t,i,j} = \barY_{m+Mt,i,j},
\]
where $(m,t) \mapsto m + tM$ is a bijection from $\{1,\dots,M\} \times \{0,\dots,T\}$ into $\{1,\dots,K\}$.
We will verify that for any $k = m + tM$,
\begin{equation}
 \label{eq:ConditionalTransitionIndicator}
 \E( \barY_{k,i,j} \, | \, \cF_{k-1})
 \weq
 \begin{cases}
 0, &\qquad t=0, \\
 X_{m,t-1}^{(i)} P_{m,i,j}, &\qquad t \ge 1.
 \end{cases}
\end{equation}
The case with $t=0$ is immediate.
Assume next that $k = m + tM$ for some $t \ge 1$.
Then, because the rows of $X$ are independent, it follows that
\begin{align*}
 \E( \barY_{k,i,j} \, | \, \cF_{k-1})
 &\weq \E\Big[ X_{m,t-1}^{(i)} X_{m,t}^{(j)} \, \big| \ X_{\ell,s} : \ell + sM < m + tM \Big] \\
 &\weq \E\Big[ X_{m,t-1}^{(i)} X_{m,t}^{(j)} \, \big| \ X_{m,s} : m + sM < m + tM \Big] \\
 &\weq \E\Big[ X_{m,t-1}^{(i)} X_{m,t}^{(j)} \, \big| \ X_{m,1}, \dots, X_{m,t-1} \Big].
\end{align*}
Because $X_{m:}$ is a Markov chain with transition matrix $P_m$,
the last expression on the right equals $X_{m,t-1}^{(i)} P_{m,i,j}$,
confirming \eqref{eq:ConditionalTransitionIndicator}.
By comparing \eqref{eq:ConditionalTransitionIndicator} with \eqref{eq:UDef} and \eqref{eq:YDef},
we find that
\[
 \barU_{k,i,j} \weq \barY_{k,i,j} - \E( \bar Y_{k,i,j} \, | \, \cF_{k-1}),
\]
so that $\E( \bar U_{k,i,j} \, | \, \cF_{k-1}) = 0$ for all $k=1,\dots,K$.
It follows that $k \mapsto \barV_{k,i,j}$, and hence also 
the vector-valued process $k \mapsto \barV_{k,i,:}$, is a martingale.

To verify the claim,
let us fix an integer $k = m + tM$.
The stated inequality is trivial when $t=0$ because then $\barU_{k,i,:} = 0$.
Assume next that $t \ge 1$.
Then \eqref{eq:UDef} implies that
\begin{align*}
 \norm{ \barU_{k,i,:} }_2^2
 &\weq X_{m,t-1}^{(i)} \sum_{j \in \Ome} \big( X_{m,t}^{(j)} - P_{m,i,j} \big)^2 \\
 &\wle X_{m,t-1}^{(i)} \sum_{j \in \Ome} \big( X_{m,t}^{(j)} + P_{m,i,j}^2 \big)
 \weq X_{m,t-1}^{(i)} \Big( 1 + \twonorm{P_{m,i,:}}^2 \Big).
\end{align*}
Because $\norm{P_{m,i,:}}_2^2 \le \norm{P_{m,i,:}}_\infty \norm{P_{m,i,:}}_1 \le 1$,
the claim follows.
\end{proof}

We will derive a concentration inequality
for $k \mapsto \barV_{k,i,:}$ by
applying a Freedman--Tropp matrix martingale concentration inequality.
To do this, we need bounds for associated predictable quadratic variation terms
\begin{align}
 \label{eq:QVCol}
 \Wcol_{K,i} &\weq \sum_{k=1}^K \E \Big[ \barU_{k,i,:} (\barU_{k,i,:})^\top \, \Big| \, \cF_{k-1} \Big], \\
 \label{eq:QVRow}
 \Wrow_{K,i} &\weq \sum_{k=1}^K \E \Big[ (\barU_{k,i,:})^\top \barU_{k,i,:} \, \Big| \, \cF_{k-1} \Big].
\end{align}
The following result provides simple formulas for these terms.

\begin{lemma}
\label{the:QV}
The predictable quadratic variation terms \eqref{eq:QVCol}--\eqref{eq:QVRow}
can be written as
\begin{align}
 \label{eq:QVColSimple}
 \Wcol_{K,i} &\weq \sum_{m=1}^M N_{m,i} (1 - \norm{P_{m,i,:}}_2^2), \\
 \label{eq:QVRowSimple}
 \Wrow_{K,i} &\weq \sum_{m=1}^M N_{m,i} \big( \diag(P_{m,i,:}) - (P_{m,i,:})^\top P_{m,i,:} \big).
\end{align}
\end{lemma}

\begin{proof}
Fix an integer $k = m + tM$ with $t \ge 1$.
By recalling the definition \eqref{eq:UDef}, we see that
\[
 \barU_{k,i,j}
 \weq X_{m,t-1}^{(i)} (X_{m,t}^{(j)} - P_{m,i,j})
\]
By noting that
\[
 \sum_{j \in \Ome} \big( X_{m,t}^{(j)} - P_{m,i,j} \big)^2
 \weq 1 - 2 P_{m,i,X_{m,t}} + \norm{P_{m,i,:}}_2^2,
\]
we see that
\begin{align*}
 \barU_{k,i,:} (\barU_{k,i,:})^\top
 \weq \sum_{j \in \Ome} \barU_{k,i,j}^2
 \weq X_{m,t-1}^{(i)} \big(1 - 2 P_{m,i,X_{m,t}} + \norm{P_{m,i,:}}_2^2 \big).
\end{align*}
Because $\E[ P_{m,i,X_{m,t}} \cond X_{m,t-1} ] = \sum_{j \in \Ome} P_{m,i,j}^2 = \norm{P_{m,i,:}}_2^2$
on the event $X_{m,t-1}=i$,
the rowwise independence of $X$ and the Markov property of $X_{m:}$ imply that
\begin{equation}
 \label{eq:QVDifferenceCol}
 \E \Big[ \barU_{k,i,:} (\barU_{k,i,:})^\top \, \Big| \, \cF_{k-1} \Big]
 \weq X_{m,t-1}^{(i)} \big(1 - \norm{P_{m,i,:}}_2^2 \big).
\end{equation}
Similarly, the entries of the matrix
$(\barU_{k,i,:})^\top \barU_{k,i,:} \in \R^{\abs{\Ome} \times \abs{\Ome}}$ satisfy
\begin{align*}
 \barU_{k,i,j} \barU_{k,i,j'}
 \weq X_{m,t-1}^{(i)} &\big( \delta_{jj'} X_{m,t}^{(j)} - X_{m,t}^{(j)} P_{m,i,j'} - X_{m,t}^{b'} P_{m,i,j} + P_{m,i,j} P_{m,i,j'} \big).
\end{align*}
The rowwise independence of $X$ and the Markov property of $X_{m:}$ imply that
\[
 \E \Big[ \barU_{k,i,j} \barU_{k,i,j'} \, \big| \, \cF_{k-1} \Big]
 \weq X_{m,t-1}^{(i)} \big( \delta_{jj'} P_{m,i,j} - P_{m,i,j} P_{m,i,j'} \big),
\]
or in matrix form,
\begin{equation}
 \label{eq:QVDifferenceRow}
 \E \Big[ \, (\barU_{k,i,:})^\top \barU_{k,i,:} \, \big| \, \cF_{k-1} \Big]
 \weq X_{m,t-1}^{(i)} \big( \diag(P_{m,i,:}) - (P_{m,i,:})^\top P_{m,i,:} \big).
\end{equation}

Recall that $\barU_{k,i,j} = 0$ for all $k = m + tM$ such that $t=0$.
Therefore, by plugging \eqref{eq:QVDifferenceCol} into \eqref{eq:QVCol}, we see that
\begin{align*}
 \Wcol_{K,i}
 \weq \sum_{k=1}^K \E \Big[ \barU_{k,i,:} (\barU_{k,i,:})^\top \, \big| \, \cF_{k-1} \Big]
 \weq \sum_{m=1}^M \sum_{t=1}^T X_{m,t-1}^{(i)} \big(1 - \norm{P_{m,i,:}}_2^2 \big),
\end{align*}
so that \eqref{eq:QVColSimple} follows.
Similarly, by plugging \eqref{eq:QVDifferenceRow} into \eqref{eq:QVRow}, we see that
\[
 \Wrow_{K,i}
 \weq \sum_{m=1}^M \sum_{t=1}^T X_{m,t-1}^{(i)} \big( \diag(P_{m,i,:}) - (P_{m,i,:})^\top P_{m,i,:} \big),
\]
which confirms \eqref{eq:QVRowSimple}.
\end{proof}

\begin{proof}[Proof of Proposition~\ref{the:TransitionFrequencies}]
We fix a state $i$ and consider $k \mapsto \barV_{k,i,:}$
defined by \eqref{eq:CumulativeProcess} as a matrix-valued random process
on state space $\R^{1 \times \abs{\Ome}}$.
Lemma~\ref{the:Martingale} implies that
this process is a martingale, with increments bounded by
$\norm{ \barU_{k,i,:} }_2 \le \sqrt{2}$ almost surely.
Lemma~\ref{the:QV} implies that the associated
quadratic variation terms \eqref{eq:QVCol}--\eqref{eq:QVRow} are given by
\begin{align*}
 \Wcol_{K,i} &\weq \sum_{m=1}^M N_{m,i} (1 - \norm{P_{m,i,:}}_2^2), \\
 \Wrow_{K,i} &\weq \sum_{m=1}^M N_{m,i} \big( \diag(P_{m,i,:}) - (P_{m,i,:})^\top P_{m,i,:} \big).
\end{align*}
The spectral norm of $\Wcol_{K,i}$, 
considered as a 1-by-1 matrix,
is bounded by
\[
 \spenorm{\Wcol_{K,i}}
 \weq \sum_{m=1}^M N_{m,i} (1 - \norm{P_{m,i,:}}_2^2)
 \wle \sum_{m=1}^M N_{m,i}
 \weq N_i.
\]
By applying Lemma~\ref{the:SpectralNorm}
and noting that $1 - 2P_{m,i,j} + \norm{P_{m,i,:}}_2^2 \le 1 + \norm{P_{m,i,:}}_2^2$ and
$\norm{P_{m,i,:}}_2 \le 1$, we see that
\begin{align*}
 \spenorm{\Wrow_{K,i}}
 \wle \sum_{m=1}^M N_{m,i}
 \left(
 \sum_{j \in \Ome} P_{m,i,j}^2 \big( 1 - 2P_{m,i,j} + \norm{P_{m,i,:}}_2^2 \big)
 \right)^{1/2}
 \wle \sqrt{2} N_i.
\end{align*}
We conclude that
$\max\{ \spenorm{ \Wrow_{K,i} }, \spenorm{ \Wcol_{K,i} } \} \wle \sqrt{2} N_i$
almost surely.

We may now apply a matrix Freedman--Tropp martingale concentration
inequality \cite[Corollary 1.3]{Tropp_2011} to conclude that for all $\eps > 0$,
\begin{align*}
 \pr\big( \spenorm{ \barV_{k,i,:} } \ge \eps, \ N_i \le s_2 \big)
 &\wle \pr\big( \spenorm{ \barV_{k,i,:}} \ge \eps, \
 \max\{ \spenorm{ \Wrow_{K,i} }, \spenorm{ \Wcol_{K,i} } \} \le \sqrt{2} s_2 \big) \\
 &\wle (1+\abs{\Ome}) \exp \Big( - \frac{\eps^2/2}{\sqrt{2} s_2 + \sqrt{2} \eps/3} \Big).
\end{align*}

By recalling \eqref{eq:CenteredTransitionCountTranslated},
we note that $\hat P_{i,:} - \tilde P_{i,:} = \frac{1}{N_i} (N_{i,:} - N_i \tilde P_{i,:}) = \frac{1}{N_i} \barV_{k,i,:}$
on the event that $N_i > 0$.
It follows that
for any $0 < s_1 \le s_2$,
\begin{align*}
 \pr( \norm{\hat P_{i,:} - \tilde P_{i,:}}_2 \ge \eps, \, s_1 \le N_i \le s_2 )
 &\weq \pr\big( \norm{ \barV_{k,i,:} }_2 \ge \eps N_i, \, s_1 \le N_i \le s_2 \big) \\
 &\wle \pr\big( \norm{ \barV_{k,i,:} }_2 \ge \eps s_1, \, N_i \le s_2 \big) \\
 &\wle (1+\abs{\Ome}) \exp \Big( - \frac{(\eps s_1)^2/2}{\sqrt{2} s_2 + \sqrt{2} \eps s_1/3} \Big) \\
 &\weq (1+\abs{\Ome}) \exp \Big( - \frac{3 \eps^2 s_1}{6 \sqrt{2} s_2/s_1 + 2 \sqrt{2} \eps} \Big).
\end{align*}
Finally, the bound $\norm{x}_1 \le \abs{\Ome}^{1/2} \norm{x}_2$ implies that
\begin{align*}
 \pr( \norm{\hat P_{i,:} - \tilde P_{i,:}}_1 \ge \eps, \, s_1 \le N_i \le s_2 )
 \wle (1+\abs{\Ome}) \exp
 \bigg( - \frac{3 \eps^2 s_1}{6 \sqrt{2} \abs{\Ome} s_2/s_1 + 2 \sqrt{2} \abs{\Ome}^{1/2} \eps} \bigg).
\end{align*}
\end{proof}

\subsection{Proofs of main results}
\label{sec:Proofs}


\begin{proof}[Proof of Theorem~\ref{the:TransitionMatrixError}]
Fix \newtext{$0 < \eps \le 1$ and $\de > 0$.}\margnote{Proof streamlined.}
Let us decompose the estimation error according to
\begin{equation}
 \label{eq:RobustMarkovEstimation1}
 \hat P - P
 \weq (\hat P - \tilde P) + ( \tilde P - P ),
\end{equation}
where $\tilde P$ is defined by \eqref{eq:MeanTransitionMatrix}.

(i) We first derive an upper bound for $\supnorm{\tilde P - P}$ on
the event $\cE = \cap_i \cE_i$ where
\begin{align*}
 \cE_i &\weq \big\{ \tfrac12 \barpi_i M T \le N_i \le \tfrac32 \barpi_i M T \big\}.
\end{align*}
Observe that\margnote{Proof argument simplified.}
\begin{align*}
 \norm{\tilde P_{i,:} - P_{i,:}}_1
 &\weq \Bignorm{ \sum_m \frac{N_{m,i}}{N_i} (P_{m,i,:} - P_{i,:}) }_1
 \wle \sum_m \frac{N_{m,i}}{N_i} \norm{P_{m,i,:} - P_{i,:}}_1
 \newtext{\wle \De_\infty.}
\end{align*}
Alternatively,
the upper bound $N_{m,i} \le T$ implies that
$\frac{N_{m,i}}{N_i} \le \frac{T}{\frac12 \barpi_i MT} = \frac{2}{\barpi_i M}$.
Therefore,
\begin{align*}
 \norm{\tilde P_{i,:} - P_{i,:}}_1
 &\wle \sum_m \frac{N_{m,i}}{N_i} \norm{P_{m,i,:} - P_{i,:}}_1
 \wle \frac{2}{\barpi_i M} \sum_m \norm{P_{m,i,:} - P_{i,:}}_1
 \newtext{\wle \frac{2 \De_1}{\barpimin}}.
\end{align*}
We conclude that
\begin{equation}
 \label{eq:RobustMarkovEstimation2}
 \supnorm{\tilde P - P}
 \wle \min \left\{ \frac{2 \De_1}{\barpimin}, \, \De_\infty \right\}
 \qquad \text{on the event $\cE$}.
\end{equation}

(ii) Define an event
$\cF = \cap_i \cF_i$, where
$
 \cF_i = \big\{ \onenorm{\hhP_{i,:} - \tilde P_{i,:}} \le \newtext{\de} \big\}.
$
The bound \eqref{eq:RobustMarkovEstimation2}
and the triangle inequality applied to \eqref{eq:RobustMarkovEstimation1},
imply that
\begin{equation}
 \label{eq:RobustMarkovEstimation3}
 \supnorm{\hat P - P}
 \wle \newtext{\de} + \min \left\{ \frac{2 \De_1}{\barpimin}, \, \De_\infty \right\}
 \qquad \text{on the event $\cE \cap \cF$}.
\end{equation}

(iii) We analyse the probabilities of $\cE$ and $\cF$.
By applying Proposition~\ref{the:StateFrequenciesInd} with $s = \frac12 \barpi_i$,
we find that
\begin{align*}
 \pr( \cE_i^c )
 \wle 2 \exp \bigg( - \frac{\gamin \barpi_i M T}{64 (1+1/(\gamin T)) + 80} + \frac12 M \eta \bigg)
 \wle \newtext{2 \exp \bigg( - \frac{\barpi_i M T'}{144} + \frac12 M \eta \bigg)}
\end{align*}
\newtext{where $T' = \frac{\gamin T}{1+1/(\gamin T)}$.}
By applying Proposition~\ref{the:TransitionFrequencies} with
$s_1 = \frac12 \barpi_i M T$ and $s_2 = \frac32 \barpi_i M T$,
we find that
\begin{align*}
 \pr( \cF_i^c \cap \cE_i )
 \wle 2 \abs{\Ome} \exp \bigg( -
 \frac{\newtext{\de^2} \barpi_i M T}
 {\newtext{a_1} \abs{\Ome} + \newtext{a_2} \abs{\Ome}^{1/2} \newtext{\de}}
 \bigg),
\end{align*}
where \newtext{$a_1 = 12 \sqrt{2}$ and $a_2 = \frac{4}{3}\sqrt{2}$}.\margnote{Proof shortened and clarified.}
Because $(\cE \cap \cF)^c = \cup_i (\cE_i^c \cup \cF_i^c)$
and
$ \pr( \cE_i^c \cup \cF_i^c )
 = \pr( \cE_i^c ) + \pr( \cF_i^c \cap \cE_i ),
$
we see by the union bound that
\begin{align*}
 \pr( (\cE \cap \cF)^c )
 &\wle \newtext{ \underbrace{2 \abs{\Ome} \exp \bigg( - \frac{\barpimin M T'}{144} + \frac12 M \eta \bigg)}_{p_1}}
 + \underbrace{2 \abs{\Ome}^2 \exp \bigg( - \frac{\newtext{\de^2} \barpi_i M T}
 {\newtext{a_1} \abs{\Ome} + \newtext{a_2} \abs{\Ome}^{1/2} \newtext{\de}}
 \bigg)}_{\newtext{p_2}}.
\end{align*}
\newtext{
We note that $p_1 \le \frac12 \eps$ and $p_2 \le \frac12 \eps$ when 
\begin{align}
 \label{eq:p1Proof}
 \barpimin M T'
 &\wge 144 \log \frac{4\abs{\Omega}}{\eps} + M \eta, \\
 \label{eq:p2Proof}
 \frac{ \de^2 \barpi_i M T}{a_1 \abs{\Ome} + a_2 \abs{\Ome}^{1/2} \de}
 &\wge \log \frac{4\abs{\Omega}^2}{\eps}.
\end{align}
}
\newtext{
Because\margnote{Proof updated to match updated theorem statement.}
$
\log \frac{4\abs{\Omega}^2}{\eps}
\le 2 \log \frac{2\abs{\Omega}}{\eps}
\le 2 \log \frac{4\abs{\Omega}}{\eps}$ due to $\eps \ge \eps^2$,
for \eqref{eq:p2Proof} it suffices that
\[
 \frac{\de^2}{a_1 \abs{\Ome} + a_2 \abs{\Ome}^{1/2} \de}
 \wge 2 \beta,
\]
where $\beta = \frac{\log \frac{4\abs{\Omega}}{\eps}}{\barpimin M T}$.
By writing $\de = \al \abs{\Omega}^{1/2} \beta^{1/2}$ in terms of
a normalised parameter $\al > 0$, we see that this 
is equivalent to
$
 \frac{\al^2}{a_1 + a_2 \beta^{1/2} \al} \ge 2.
$
In particular, when \eqref{eq:p1Proof} holds, we see that $\beta \le \frac{1}{144}$ due to $T' \le T$,
and therefore $\frac{\al^2}{a_1 + a_2/\sqrt{144} \al} \ge 2$ is sufficient
for \eqref{eq:p2Proof}. Numerically we may check that $\al = 7$ suffices.
By \eqref{eq:RobustMarkovEstimation3},
\[
 \supnorm{\hat P - P}
 \wle 7 \sqrt{\abs{\Omega}\beta} + 2 \min \left\{ \frac{\De_1}{\barpimin}, \, \De_\infty \right\}
 \qquad \text{on the event $\cE \cap \cF$}
\]
that holds with probability at least $1-\eps$, whenever \eqref{eq:p1Proof} is valid.
This confirms the statement of the theorem, and we may choose $C_1 = 7$, $C_2=2$, and $C_3 = 144$
as the universal constants.
}
\end{proof}

\begin{proof}[Proof of Theorem~\ref{the:StationaryDistributionError}]
Fix \newtext{$\de, \eps \in (0,1]$}.
By applying Proposition~\ref{the:StateFrequenciesInd} with $s=\newtext{\de}$, we find that
\begin{align*}
 \pr( \abs{\hat \pi_i - \barpi_i} \ge \newtext{\de} )
 \wle 2 \exp \bigg( - \frac{\gamin MT \newtext{\de^2}}{16 (1+1/(\gamin T)) \barpi_i + 40 \newtext{\de}} + M \eta \bigg)
 \wle \newtext{2 \exp \bigg( - \frac{MT' \newtext{\de^2}}{56} + \frac12 M \eta \bigg)},
\end{align*}
\newtext{where $T' = \frac{\gamin T}{1+1/(\gamin T)}$.}
\newtext{
Then by the union bound,\margnote{Proof updated to match updated theorem statement.}
\[
 \pr( \supnorm{\hat \pi - \barpi} \ge \newtext{\de} )
 \wle 2 \abs{\Omega} \exp \bigg( - \frac{MT' \newtext{\de^2}}{56} + M \eta \bigg).
\]
The right side of the above inequality is at most $\eps$ when
$\frac{MT' \newtext{\de^2}}{56} - M \eta \ge \log \frac{2\abs{\Omega}}{\eps}$.
Therefore, $\supnorm{\hat \pi - \barpi} < \de$ with probability at least $1-\eps$,
when we set 
$\de = C \sqrt{ \frac{\log (2\abs{\Omega}/\eps) + M\eta}{MT'} }$ with $C=\sqrt{56}$.

In the above derivation we implicitly assumed that
$C \sqrt{ \frac{\log (2\abs{\Omega}/\eps) + M\eta}{MT'} } \le 1$;
if this is not the case, then the claim of the theorem is trivial.
}
\end{proof}


\begin{proof}[Proof of Theorem~\ref{the:TransitionMatrixErrorCorrupted}]
The state and transition frequencies in \eqref{eq:MarkovFrequencies} 
of the data matrix can be decomposed as
$N_i = N_i^{(0)} + N_i^{(1)}$ and
$N_{i,j} = N_{i,j}^{(0)} + N_{i,j}^{(1)}$, where
\[
 N_i^{(k)} = \sum_{m \in \cM_k} \sum_{t=1}^{T} X_{m,t-1}^{(i)}
 \qquad\text{and}\qquad
 N_{i,j}^{(k)} = \sum_{m \in \cM_k} \sum_{t=1}^T X_{m,t-1}^{(i)} X_{m,t}^{(j)},
\]
and $\cM_0$ and $\cM_1$ denote the sets of uncorrupted and corrupted rows, respectively.
Analogously with \eqref{eq:EmpiricalTransitionMatrix} and \eqref{eq:EmpiricalDistribution},
define the empirical transition matrix and the empirical distribution
of the uncorrupted rows by
\[
 \hat P^{(0)}_{i,j}
 =
 \begin{cases}
   \frac{N^{(0)}_{i,j}}{N^{(0)}_i} &\quad \text{if $N^{(0)}_i > 0$}, \\
   \frac{1}{\abs{\Ome}} &\quad \text{else},
 \end{cases}
 \qquad \text{and} \qquad
 \hat\pi^{(0)}_i
 = \frac{N^{(0)}_i}{M_0 T}.
\]

(i) Fix a state $i$ such that $\newtext{N_i^{(0)}} > 0$.\margnote{Corrected $N_i \mapsto N_i^{(0)}$}
Then
\[
 \hat P_{i,j} - P_{i,j}
 \weq \frac{N^{(0)}_{i,j}}{N_i} + \frac{N^{(1)}_{i,j}}{N_i} - P_{i,j}
 \weq \frac{N^{(0)}_i}{N_i} (\hat P^{(0)}_{i,j} - P_{i,j})
 + \frac{N^{(1)}_{i,j}}{N_i} - \frac{N^{(1)}_i}{N_i} P_{i,j}.
\]
It follows that 
\[
 \abs{\hat P_{i,j} - P_{i,j}}
 \wle \abs{\hat P^{(0)}_{i,j} - P_{i,j}} + \frac{N^{(1)}_{i,j}}{N_i} + \frac{N^{(1)}_i}{N_i} P_{i,j}.
\]
By summing over $j$, we find that
\[
 \onenorm{\hat P_{i,:} - P_{i,:}}
 \wle \onenorm{\hat P^{(0)}_{i,:} - P_{i,:}} + 2 \frac{N^{(1)}_i}{N_i}.
\]
Also, by noting that $N_i = M_0 T \hat\pi^{(0)}_i + N^{(1)}_i$ and $N^{(1)}_i \le M_1 T$, it follows that
\[
 \frac{N^{(1)}_i}{N_i}
 \weq \frac{N^{(1)}_i}{M_0 T \hat\pi^{(0)}_i + N^{(1)}_i}
 \wle \frac{M_1 T}{M_0 T \hat\pi^{(0)}_i + M_1 T}
 \wle \frac{M_1 T}{M T \hat\pi^{(0)}_i}
 \weq \frac{M_1/M}{\hat\pi^{(0)}_i}.
\]
By combining the above two inequalities, we conclude that
\newtext{when $\hat\pi^{(0)}_{\rm min} > 0$},\margnote{Shortened and clarified.}
\begin{equation}
 \label{eq:TransitionMatrixErrorCorrupted1}
 \newtext{
 \supnorm{\hat P - P} \wle \supnorm{\hat P^{(0)} - P}
 + 2 \frac{M_1/M}{\hat\pi^{(0)}_{\rm min}}.
 }
\end{equation}

(ii) We apply the earlier concentration results to the submatrix $(X_{m,t})_{m \in \cM_0, t \le T}$ 
consisting of the uncorrupted rows, for which $\hat P^{(0)}$ is the empirical transition matrix
and $\hat\pi^{(0)}$ is the empirical distribution.
Then we may apply Theorems~\ref{the:TransitionMatrixError} and~\ref{the:StationaryDistributionError}
verbatim to the data matrix $(X_{m,t})_{m \in \cM_0, t \le T}$. 
\newtext{
By Theorem~\ref{the:TransitionMatrixError},\margnote{Proof adapted to match the updated theorem statement}
there exist universal constants \(\tC_1, \tC_2, \tC_3 > 0\) such that
\begin{equation}
 \label{eq:TransitionMatrixErrorNewForCorrupted}
 \supnorm{\hhP^0 - P}
 \wle \tC_1 \sqrt{ \frac{\abs{\Omega} \log(8\abs{\Omega}/\eps)}{\barpimin^{(0)} M_0 T} }
 + \tC_2 \min \left\{ \frac{\De_1^{(0)}}{\barpimin^{(0)}}, \, \De_\infty^{(0)} \right\}
\end{equation}
with probability at least $1-\eps/2$,
provided that
\begin{equation}
 \label{eq:TransitionMatrixErrorLargeSampleForCorrupted}
 M_0 T' \wge \tC_3 \frac{\log(8\abs{\Omega}/\eps) + M_0 \eta^{(0)}}{\barpimin^{(0)}}.
\end{equation}
\newtext{By Theorem~\ref{the:StationaryDistributionError},}
there exists a universal constant \(\tC_4 > 0\) such that
\begin{equation}
 \label{eq:StationaryDistributionErrorNewForCorrupted}
 \supnorm{\hat \pi^{(0)} - \bar\pi^{(0)}}
 \wle \tC_4 \sqrt{ \frac{\log(4\abs{\Omega}/\eps) + M_0 \eta^{(0)}}{M_0 T'} }
\end{equation}
with probability at least $1-\eps/2$.

(iii) Let us now fix a constant $C > 0$, and consider the assumption
\begin{equation}
 \label{eq:TransitionMatrixErrorLargeSampleForCorruptedStrong}
 M_0 T' \wge C \frac{\log(8\abs{\Omega}/\eps) + M_0 \eta^{(0)}}{\big(\barpimin^{(0)}\big)^2}.
\end{equation}
We note that \eqref{eq:TransitionMatrixErrorLargeSampleForCorruptedStrong}
implies \eqref{eq:TransitionMatrixErrorLargeSampleForCorrupted}
when $C \ge \tC_3$.
We also that that under 
\eqref{eq:TransitionMatrixErrorLargeSampleForCorruptedStrong} and
\eqref{eq:StationaryDistributionErrorNewForCorrupted},
\begin{align*}
 \supnorm{\hat \pi^{(0)} - \bar\pi^{(0)}}
 \wle \tC_4 \sqrt{ \frac{\log(4\abs{\Omega}/\eps) + M_0 \eta^{(0)}}{M_0 T'} }
 \wle \frac{\tC_4 \barpimin^{(0)}}{\sqrt{C}}
 \wle \frac12 \barpimin^{(0)},
\end{align*}
when $C \ge 2 \tC_4^2$.
We define $C_4 = \max\{\tC_3, 2 \tC_4^2\}$
and assume from now on that \eqref{eq:TransitionMatrixErrorLargeSampleForCorruptedStrong} holds for $C = C_4$.

(iv) Let $\cA$ be the event that
\eqref{eq:TransitionMatrixErrorNewForCorrupted} and
\eqref{eq:StationaryDistributionErrorNewForCorrupted} both hold.
Then \eqref{eq:TransitionMatrixErrorLargeSampleForCorrupted} is valid,
and we conclude by the union bound that $\pr(\cA) \ge 1 - \eps$.
On the event $\cA$, 
we see that
$\hat\pi_{\rm min}^{(0)} \ge \barpimin^{(0)} - \supnorm{\hat \pi^{(0)} - \barpi^{(0)}} \ge \frac12 \barpimin^{(0)}$,
so that $\hat\pi_{\rm min}^{(0)} > 0$.
Therefore, the inequalities \eqref{eq:TransitionMatrixErrorCorrupted1}, 
\eqref{eq:TransitionMatrixErrorNewForCorrupted}, and 
\eqref{eq:StationaryDistributionErrorNewForCorrupted} are all valid
on the event $\cA$, and we conclude
applying
$2 \frac{M_1/M}{\hat\pi^{(0)}_{\rm min}} \le 4 \frac{M_1/M}{\barpimin^{(0)}}$
that
\[
 \supnorm{\hat P - P}
 \wle \tC_1 \sqrt{ \frac{\abs{\Omega} \log(8\abs{\Omega}/\eps)}{\barpimin^{(0)} M_0 T} }
 + \tC_2 \min \left\{ \frac{\De_1^{(0)}}{\barpimin^{(0)}}, \, \De_\infty^{(0)} \right\}
 + 4 \frac{M_1/M}{\barpimin^{(0)}}.
\]
This confirms that 
Theorem~\ref{the:TransitionMatrixErrorCorrupted}
is valid with universal constants
$C_1 = \tC_1$, 
$C_2 = \tC_2$, 
$C_3 = 4$,
and
$C_4 = \max\{\tC_3, 2 \tC_4^2\}$.

By noting that in Theorem 3.1 we may take $(\tC_1,\tC_2,\tC_3) = (7,2,144)$
and in Theorem 3.2 we may take $\tC_4 = 8$,
we see that here we may take $(C_1,C_2,C_3,C_4) = (7,2,4,144)$.
}
\end{proof}

\subsection{Auxiliary technical results}
\label{sec:Technical}

\subsubsection{Information theory}
\label{sec:InformationTheory}

The \emph{\Renyi divergence} of order $\al > 1$ is defined by the formula
$D_\al(P \| Q) = (\al-1)^{-1} \int ( \frac{dP}{dQ} )^\al dQ$ when $P$ is absolutely continuous
against $Q$, and $D_\al(P \| Q) = \infty$ otherwise.
\Renyi divergences and $\chi^2$-divergences are related by the formula
$D_2(P \| Q) = \log( 1 + \chi^2(P \| Q) )$, 
see \cite{Polyanskiy_Wu_2024,VanErven_Harremoes_2014} for details.

\begin{lemma}
\label{the:RenyiBound}
For any probability measures $P,Q$, any $\al > 1$, and any measurable set $A$,
\[
 P(A)
 \wle e^{(1-1/\al) D_\al(P \| Q)} \, Q(A)^{1-1/\al}.
\]
\end{lemma}
\begin{proof}
We assume that $P$ is absolutely continuous with respect to $Q$,
because otherwise $D_\al(P \| Q) = \infty$ and the claim is trivial.
\Holder's inequality applied to the Radon--Nikodym derivative $\frac{dP}{dQ}$ implies that
\begin{align*}
 P(A)
 \weq \int \frac{dP}{dQ} 1_A \, dQ
 \wle \left( \int \left( \frac{dP}{dQ}\right)^\al dQ \right)^{1/\al} \left( \int 1_A^\be dQ \right)^{1/\be},
\end{align*}
where $\frac{1}{\be} = 1 - \frac{1}{\al}$.
Because
$
 \int ( \frac{dP}{dQ} )^\al dQ
 = e^{(\al-1) D_\al(P \| Q)}
$
and $ \int 1_A^\be dQ = Q(A)$,
the claim follows.
\end{proof}

\begin{lemma}
\label{the:NonstationaryMC}
Let $\pr_\mu$ (resp.\ $\pr_\nu$) be the law of a random matrix $(X_{m,t})$ with independent rows, in which
each row $m$ is a Markov chain with transition matrix $P_m$ and initial distribution $\mu_m$ (resp.\ $\nu_m$).
Then for any $\cA \subset \Omega^{M \times (T+1)}$, and any $\al>1$,
\[
 \pr_\mu(\cA)
 \wle e^{(1 - \frac{1}{\al} ) \sum_m D_\al( \mu_m \| \nu_m )} \ \pr_\nu(\cA)^{1 - \frac{1}{\al}}.
\]
\end{lemma}
\begin{proof}
We may write 
$\pr_\mu = \bigotimes_m f_m$ and $\pr_\nu =  \bigotimes_m g_m$, where $f_m,g_m$ denote path distributions
of the Markov chains, with probability mass functions
\begin{align*}
 f_m(x_0,\dots, x_T) &\weq \mu_m(x_0) P_m(x_0,x_1) \cdots P_m(x_{T-1}, x_T), \\
 g_m(x_0,\dots, x_T) &\weq \nu_m(x_0) P_m(x_0,x_1) \cdots P_m(x_{T-1}, x_T).
\end{align*}
Because \Renyi divergences tensorise
\citep{Leskela_2024,VanErven_Harremoes_2014},
we find that
$
 D_\al( \pr_\mu \| \pr_\nu )
 = \sum_m D_\al( f_m \| g_m ).
$
Because
$\frac{f_m(x_0, \dots, x_T)}{g_m(x_0, \dots, x_T)} = \frac{\mu_m(x_0)}{\nu_m(x_0)}$,
we also find that
$
 D_\al( f_m \| g_m )
 = D_\al( \mu_m \| \nu_m ).
$
Hence
$
 D_\al( \pr_\mu \| \pr_\pi )
 = \sum_m D_\al( \mu_m \| \nu_m ),
$
and the claim follows by Lemma~\ref{the:RenyiBound}.
\end{proof}

\subsubsection{Linear algebra}

\begin{lemma}
\label{the:SpectralNorm}
For any $u \in \R^n$, 
the spectral norm of $U = \diag(u) - u u^\top$
is bounded by
$\spenorm{U}^2 \le \sum_{i=1}^n u_i^2 (1 - 2 u_i + \norm{u}_2^2 )$.
\end{lemma}
\begin{proof}
Denote by $e_i$ the $i$-th standard basis vector of $\R^n$.
Note that
\begin{align*}
 \norm{Ux}_2^2
 \weq \sum_{i=1}^n \big(u_i x_i - u_i (u^\top x) \big)^2
 \weq \sum_{i=1}^n u_i^2 \big( (e_i - u)^\top x \big)^2.
\end{align*}
Then by the Cauchy--Schwarz inequality,
$
 \norm{Ux}_2^2
 \le \sum_{i=1}^n u_i^2 \norm{e_i - u}_2^2 \norm{x}_2^2.
$
Because the latter inequality holds for all $x$,
we conclude that
\[
 \spenorm{U}^2
 \wle \sum_{i=1}^n u_i^2 \norm{e_i - u}_2^2
 \weq \sum_{i=1}^n u_i^2 (1 - 2 u_i + \norm{u}_2^2 ).
\]
\end{proof}

\paragraph{Acknowledgments.}
\newtext{The\margnote{Thank you!} authors thank the anonymous referee for constructive comments that have improved the quality of this article.}

\ifarxiv
\bibliographystyle{apalike}
\footnotesize
\fi

\bibliography{lslReferences}

\end{document}